\documentclass[12pt]{amsart}

\usepackage{amsmath,amsthm,amssymb,amsfonts}
\usepackage{mathtools}
\usepackage{bm} \usepackage{upgreek} \usepackage{pifont} \usepackage{blkarray} \usepackage[all]{xy} \usepackage{enumitem} \usepackage{subdepth} \usepackage{fouridx} 
\usepackage{graphicx}
\usepackage{tikz}
\usepackage{calc}
\usetikzlibrary{calc}
\usetikzlibrary{positioning}
\usetikzlibrary{decorations.pathmorphing,decorations.markings}
\usetikzlibrary{arrows.meta}
\usetikzlibrary{backgrounds}
\usetikzlibrary{fit}

\usepackage[margin = 1.2in]{geometry}

\usepackage[colorlinks,linkcolor=red,anchorcolor=blue,citecolor=green]{hyperref}
\usepackage[capitalise,noabbrev]{cleveref}

\newtheorem{theorem}{Theorem}

\newtheorem{corollary}[theorem]{Corollary}

\newtheorem{lemma}[theorem]{Lemma}

\newtheorem{proposition}[theorem]{Proposition}

\DeclareMathOperator{\rank}{rank}

\DeclarePairedDelimiter{\set}{\{}{\}}
\DeclarePairedDelimiter{\abs}{\vert}{\vert}
\DeclarePairedDelimiter{\floor}{\lfloor}{\rfloor}

\begin{document}

\title[Spherical embeddings of symmetric association schemes in $\mathbb{R}^3$]{Spherical embeddings of symmetric association schemes in 3-dimensional 
Euclidean space}

\author{Eiichi Bannai \and Da Zhao}
\address{School of Mathematical Sciences, Shanghai Jiao Tong University, Shanghai, China.}
\email{\{bannai,jasonzd\}@sjtu.edu.cn}



\begin{abstract}
We classify the symmetric association schemes with faithful spherical embedding in 3-dimensional Euclidean space. Our result is based on previous research on primitive association schemes with $m_1 = 3$.

\end{abstract}

\maketitle

\section{Introduction}

Let $\mathfrak X=(X,\{R_i\}_{0\leq i\leq d})$ be a symmetric association schemes, 
and let $A_i$ be the adjacency matrix of the relation $R_i$ and let 
$E_i~  (0\leq i\leq d)$ be the primitive idempotents. The spherical embedding of a 
symmetric association scheme $\mathfrak X$ with respect to $E_i$ is the mapping: $X \rightarrow \mathbb R^{m_i}$ defined by
$$x\rightarrow \overline{x}=\sqrt{\frac{|X|}{m_i}}E_i\phi_{x},$$ 
where $\phi_{x}$ is the characteristic vector of $x$ (regarded as a column vector 
of size $|X|$) and $m_i=\rank E_i.$ 
Then the $\overline{x}$ are all on the unit sphere $S^{m_i-1}\subset \mathbb R^{m_i}.$ 
In what follows, we identify $\overline{X}$ and $X$ when the embedding is faithful.
The reader is referred to \cite{MR2212140,BBI16,MR882540,MR1002568} for the basic concept of association schemes 
and spherical embeddings of association schemes.

In \cite{MR2212140}, Bannai-Bannai studied the spherical embeddings of symmetric association 
schemes with $m_1=3,$ i.e., in $\mathbb R^3$, and determined that there exists only one such faithful spherical 
embedding if we assume the association scheme is primitive. Namely, it must be a regular tetrahedron, i.e., the association scheme with $d=1$ corresponding to the complete 
graph $K_4$. On the other hand, in \cite{MR2212140} it was known that the method used there 
could be applied to study imprimitive association schemes as well, but it was left unanswered. 
The proof in \cite{MR2212140} is of completely elementary geometric nature, and it has close 
connections with the classification of regular polyhedrons and quasi-regular 
polyhedrons, etc.. We first remark that the method in \cite{MR2212140} 
essentially proves the following result.

\begin{proposition} \label{pro:Adapter}
	Let $X$ be a spherical embedding of $\mathfrak X.$ 
	Let $A(X)=\{\langle x,y \rangle \mid x,y \in X, x\not= y\}$ 
	and $\alpha =\max A(X),$ where $\langle x,y \rangle$ is 
	the usual inner product on $\mathbb R^n$. 
	Suppose $m_1=3$ and that the spherical embedding of $X$ is faithful, i.e., $1\not\in A(X)$.
	Moreover, we assume that the relation $R_1$ of the association scheme is contained in the maximum inner product relation $\Gamma_\alpha = \{ (x,y) \mid \langle x,y \rangle=\alpha\}$. Then 
	\begin{enumerate}		\item The valency of the graph $(X, \Gamma_\alpha)$ is at most $5$. Consequently, $k_1 =1,2,3,4,$ or $5$ where $k_1$ is the valency of the graph $(X, R_1)$.
		\item If we further assume $R_1 = \Gamma_\alpha$, we can show that $X\subset S^2$ is as follows.
		\begin{enumerate}			\item If $k_1=5,$ then each connected component of $(X,R_1)$ is the regular icosahedron $(|X|=12)$.
			\item If $k_1=4,$ then each connected component of $(X,R_1)$ is the regular octahedron $(|X|=6)$, the quasi-regular polyhedron of type $[3,4,3,4]$ $(|X|=12)$, or the quasi-regular polyhedron of type $[3,5,3,5]$ $(|X|=30)$.
			\item If $k_1=3,$ then each connected component of $(X,R_1)$ is the regular tetrahedron $(|X|=4)$, the cube $(|X|=8)$, or the regular dodecahedron $(|X|=20)$.
		\end{enumerate}
	\end{enumerate}
\end{proposition}

If we look at the proof given in \cite{MR2212140} carefully, it is in fact possible to see that the statements as given in \cref{pro:Adapter} do hold.

\begin{corollary} \label{coro:Dentist}
Let $\mathfrak X$ be a Q-polynomial association scheme 
spherically embedded in $\mathbb R^3,$ then such $\mathfrak X$ are classified. 
(They are in a part of those in the list of \cref{thm:Rest} above.) 
\end{corollary}

This can be obtained from \cref{pro:Adapter}, since if $\mathfrak X$ is a Q-polynomial association scheme, then the $Q_1(i) ~ (0\leq i\leq d)$ are all distinct, so the assumption of \cref{pro:Adapter} is satisfied. 
(We believe that the result obtained in \cref{coro:Dentist} should be expected to be known, 
but it seems that this was not explicitly mentioned in the literature, as far as we could check.)

If we can show that (1) $k_1 \neq 1$ and $k_1 \neq 2$, (2) $\Gamma_\alpha$ do not split into more than one relations, and (3) the graph $(X,R_1)$ is connected, then we completely classify symmetric association schemes with $m_1 = 3$.
Now we are able to deal with these difficulties completely. Our main theorem is stated as follows.

\begin{theorem} \label{thm:Rest}
Let $\mathfrak X$ be a symmetric association scheme. If $\mathfrak X$ has a faithful spherical embeddings with $m_1=3$, then it must be one of the followings: the regular tetrahedron $(\abs{X}=4)$, the regular octahedron $(\abs{X}=6)$, the cube $(\abs{X}=8)$, the regular icosahedron $(\abs{X}=12)$, the quasi-regular polyhedron of type $[3,4,3,4]$ $(\abs{X}=12)$, the regular dodecahedron $(\abs{X}=20)$ and the quasi-regular polyhedron of type $[3,5,3,5]$ $(\abs{X}=30)$.
\end{theorem}

\section{Proofs}

	In the following $X$ is a faithful spherical embedding of a symmetric association scheme $\mathfrak X=(X,\{R_i\}_{0\leq i\leq d})$. 
	Let $A(X) = \set*{ \langle x,y \rangle \mid x,y \in X, x \neq y}$ and $\alpha = \max A(X)$. 
	We call $\Gamma_\alpha = \set*{ (x,y) \mid x,y \in X, \langle x,y\rangle = \alpha}$ the maximum inner product relation.

	\begin{lemma} \label{lem:Stool}
		Let $X$ be a faithful spherical embedding of a symmetric association $\mathfrak X=(X,\{R_i\}_{0\leq i\leq d})$ with respect to the idempotent $E_1$. Let $R_1 \neq R_0$ be a relation in $\mathfrak{X}$ with valency $k_1 = 1$. Then $R_1$ gives an antipodal relation in $X$. 
	\end{lemma}

	\begin{proof}
		Since $k_1 = 1$, the adjacency matrix $A_1$ of $R_1$ must satisfy $A_1^2 = A_0$. We apply $E_i$ on both sides, and we have $P^2_1(1) = P_0(1) = 1$. Therefore $\dfrac{Q_1(1)}{m_1} = \dfrac{\overline{P_1(1)}}{k_1} = \pm 1$. Since the embedding is faithful, $\dfrac{Q_1(1)}{m_1}$ has to be $-1$, i.e., $R_1$ gives an antipodal relation in $X$. 
	\end{proof}

	\begin{corollary} \label{coro:Moonscape}
		Let $X$ be a faithful spherical embedding of a symmetric association $\mathfrak X=(X,\{R_i\}_{0\leq i\leq d})$ with respect to the idempotent $E_1$. Let $R_1 \neq R_0$ be a relation in $\mathfrak{X}$ with valency $k_1 = 1$, and $R_1 \subseteq \Gamma_\alpha$. Then $\abs{X}=2$.
	\end{corollary}

	\begin{lemma} \label{lem:Salmon}
		Let $X$ be a faithful spherical embedding of a symmetric association $\mathfrak X=(X,\{R_i\}_{0\leq i\leq d})$ with respect to the idempotent $E_1$. Let $R_1 \neq R_0$ be a relation in $\mathfrak{X}$ with valency $k_1 = 2$, and $R_1 \subseteq \Gamma_\alpha$. Then $(X,R_1)$ is a disjoint union of regular $\ell$-gons on big circles, where $\ell$ is a positive integer.
	\end{lemma}

	\begin{proof}
		Since $k_1 = 2$, we may assume the adjacency matrix $A_1$ of the relation $R_1$ have the following form:
		\[
			A_1 = 
				\begin{bmatrix}
					C_1 	& 0		& \dots 	& 0 	\\
					0		& C_2 	& \dots 	& 0 	\\
					\vdots	& \vdots& \ddots 	& \vdots 	\\
					0		& 0		& \ldots			& C_N 	
				\end{bmatrix}
		\]
		where every $C_r$ is the adjacency matrix of a cycle graph for $0 \leq r \leq N$. Firstly, we show that they are of the same length. Let $C_1$ be a cycle of the minimum length $\ell_1$ among these cycles. Suppose there exists a cycle of length $\ell_2 > \ell_1$, say $C_2$. We apply  the characteristic polynomial $\chi$ of $C_1$ to $A_1$. Since $\chi(C_1) = 0$ and $\chi(C_2) \neq 0$, we know that $\chi(A_1)$ is of the following form:
		\[
			\chi(A_1) = 
				\begin{bmatrix}
					0 		& 0		& \dots 	& 0 	\\
					0		& \chi(C_2) 	& \dots 	& 0 	\\
					\vdots	& \vdots& \ddots 	& \vdots 	\\
					0		& 0		& \dots			& \chi(C_N) 	
				\end{bmatrix}
		\]
		Note that in the Bose-Mesner algebra, two rows of a matrix only differ by a permutation. Hence $\chi(A_1)$ is the zero matrix, contradiction. So all the cycles are of the same length, say $\ell$. 

		Let $A_1, A_2, \dots, A_{\floor{\ell/2}}$ be the adjacency matrices of distance relations with respect to $A_1$ and let $k_1, k_2, \dots, k_{\floor{\ell/2}} \leq 2$ be their valencies. We claim that they are indeed some adjacency matrices of relations in $\mathfrak{X}$. 
		Suppose otherwise, say $A_i$ splits into two relations $A_i = A_{i_1} + A_{i_2}$ for some $1 \leq i \leq \floor{\ell/2}$. 
		We must have $k_{i_1} = k_{i_2} =1$.
		By \cref{lem:Stool}, they are both two antipodal relations, which contradicts to the faithful condition.
		
		Next we show that the points of $X$ form regular $\ell$-gons in the embedding. 
		Again by applying $E_1$ to both sides of  $\chi(A_1) = 0$, we obtain that $P_1(1)$ must be among the eigenvalues of the $\ell$-cycle graph, i.e., $P_1(1) \in \set*{2\cos\frac{2t \pi}{\ell} \mid t = 1, \dots, \floor{\ell/2}}$. 
		The distance matrices satisfy the recurrence relations $A_1 A_i = b_{i-1} A_{i-1} + a_i A_i + c_{i+1} A_{i+1}$ for all $1 \leq i \leq \floor{\ell/2}$, where the intersection numbers are given by
		\[
			\begin{bmatrix}
				- & c_1 & \dots & c_{\floor{\ell/2}-1} & c_{\floor{\ell/2}} \\
				a_0 & a_1 & \dots & a_{\floor{\ell/2}-1} & a_{\floor{\ell/2}} \\
				b_0 & b_1 & \dots & b_{\floor{\ell/2}-1} & - 	
			\end{bmatrix}
			=
			\begin{bmatrix}
				- & 1 & \dots & 1 & 2 \\
				0 & 0 & \dots & 0 & 0 \\
				2 & 1 & \dots & 1 & - 	
			\end{bmatrix}
		\]
		if $\ell$ is even, and
		\[
			\begin{bmatrix}
				- & c_1 & \dots & c_{\floor{\ell/2}-1} & c_{\floor{\ell/2}} \\
				a_0 & a_1 & \dots & a_{\floor{\ell/2}-1} & a_{\floor{\ell/2}} \\
				b_0 & b_1 & \dots & b_{\floor{\ell/2}-1} & - 	
			\end{bmatrix}
			=
			\begin{bmatrix}
				- & 1 & \dots & 1 & 1 \\
				0 & 0 & \dots & 0 & 1 \\
				2 & 1 & \dots & 1 & - 	
			\end{bmatrix}
		\]
		if $\ell$ is odd.

		Now we focus on the odd cases, and the even cases are similar.
		Suppose $P_1(1) = 2\cos\frac{2t \pi}{l}$ for some $t \in \set{1, 2, \dots, \floor{\ell/2}}$. By applying $E_1$ to the recurrence relations, we can calculate the eigenvalues $P_i(1)$ for all $1 \leq i \leq \floor{\ell/2}$. 
		$P_1(1) P_1(1) = P_2(1) + 2 P_0(1)$ implies $P_2(1) = 2\cos\frac{2 \cdot 2t \pi}{\ell}$ and $P_1(1) P_2(1) = P_1(1) + P_3(1)$ implies $P_3(1) = 2\cos\frac{3 \cdot 2t \pi}{\ell}$. 
		Recursively we have $P_i(1) = 2\cos\frac{i \cdot 2t \pi}{\ell}$ where $i \in \set{1,2,\dots,\floor{l/2}}$. 
						In this way we obtain all the inner products of points in an $\ell$-cycle.
		\[
			\frac{Q_1(i)}{m_1} = \frac{\overline{P_i(1)}}{k_i} = \cos\frac{i \cdot 2t \pi}{\ell}.
		\]
				
		Let $0 < \theta \leq \pi$ be the angle such that $\cos\frac{ 2t \pi}{\ell} = \cos \theta$. Let $X_1, X_2, \dots, X_\ell$ be the points in $X$ which correspond to an $\ell$-cycle (ordered by the cycle) and let $O$ be the origin. Then $\angle X_1 O X_2 = \angle X_2 O X_3 = \theta$ and $\measuredangle X_1 O X_3 = 2\theta$ or $\measuredangle X_3 O X_1 = 2\theta$. So $O, X_1, X_2, X_3$ have to be coplane. Similarly $O, X_2, X_3, X_4$ are coplane. Hence $X_1, X_2, \dots, X_\ell$ are on a big circle. Since the embedding is faithful, we obtain that $\gcd(t,\ell) = 1$. What's more, the parameter $t$ has to be $1$, because $R_1$ is a subset of the maximum inner product relation $\Gamma_\alpha$, i.e., $\frac{Q_1(1)}{m_1}$ is the maximum among the inner products. Therefore, $(X,R_1)$ is a disjoint union of regular $\ell$-gons on big circles.
	\end{proof}

	\begin{corollary} \label{coro:Garbage}
		Let $X$ be a faithful spherical embedding of a symmetric association $\mathfrak X=(X,\{R_i\}_{0\leq i\leq d})$ with respect to the idempotent $E_1$ with $m_1 = 3$. Let $R_1 \neq R_0$ be a relation in $\mathfrak{X}$, and $R_1 \subseteq \Gamma_\alpha$. Then $k_1 \neq 2$.
	\end{corollary}

	\begin{proof}
		Suppose $k_1 = 2$, we've proved in \cref{lem:Salmon} that $(X,R_1)$ is a disjoint union of regular $\ell$-gons on big circles, where $\ell$ is a positive integer. Now we use the condition $m_1 = 3$, i.e., the embedding is on the unit sphere $S^2$. Note that every two big circles on $S^2$ intersect, hence there could be only one regular $\ell$-gon. Otherwise it would contradict to the maximum inner product relation. So the association scheme is nothing but a regular $\ell$-gon, however there is no idempotent with rank $3$ in such association scheme.
	\end{proof}

	Now we are able to prove the main theorem.

	\begin{proof}[Proof of \cref{thm:Rest}]
		Let $R_1 \neq R_0$ be a relation in $\mathfrak X$ and $R_1 \subseteq \Gamma_\alpha$. By \cref{coro:Moonscape,coro:Garbage}, we know that $k_1$ cannot be $1$ or $2$. By \cref{pro:Adapter}, we also know that the valency of $\Gamma_\alpha$ is at most 5, therefore $k_1 = 3, 4, 5$ and $R_1 = \Gamma_\alpha$. 
				Again by \cref{pro:Adapter}, each connected component of $(X,R_1)$ is in the list. So we only need to prove that there couldn't be two or more connected components. This is indeed true because each one in the list satisfy the following property: Adding another point to it on $S^2$ would result in a strict increase of the maximum inner product among points.
	\end{proof}

\section{Concluding Remarks}

\begin{enumerate}
	\item This paper would be interesting as an interplay of the theory of association schemes and the elementary geometric considerations in discrete geometry. 
	Association schemes can be a more standard tool to study good geometric structures such as regular polyhedron, quasi-regular polyhedrons, as well as similar or more general objects in higher dimensions. 
	\item There are some considerable differences between determining all symmetric association schemes with $m_1=3$ and determining all faithful spherical embeddings with $m_1=3$ of symmetric association schemes. For example, van Dam, Koolen, Park \cite[Section 2.5, page 6]{1701.03193} describes the difficulty of the former problem. On the other hand, from the geometric point, the most crucial problem would be the latter one that is answered in this paper.
	\item It would be very interesting to study spherical embeddings of symmetric association schemes with $m_1=4$. In particular, it would be interesting to try to classify (primitive) Q-polynomial association schemes which are spherically embedded with $m_1=4.$ We hope that our method for $m_1=3$ is somehow useful for that. On the other hand, the complete classification of faithful spherical embedding of symmetric association schemes with $m_1=4$ seems to be still somehow distant, as there are infinitely many such examples. 
\end{enumerate}

\section*{Acknowledgments}
This research was supported in part by NSFC Grant 11271257 and 11671258.

\bibliographystyle{plain}
\bibliography{m1=3}

\end{document}